%% file: ms.tex
\documentclass[10pt,letterpaper]{amsart}

\usepackage[margin=1.5in]{geometry}

\usepackage{standalone}

\usepackage{amsrefs}
\usepackage{amsfonts}
\usepackage{amssymb}
\usepackage{amsmath}
\usepackage{amsthm}
\usepackage[inline]{enumitem}
\usepackage{hyperref}
\usepackage{enumitem}

\usepackage{color}

\usepackage[textsize=small]{todonotes}
\usepackage{soul}
\makeatletter
 \if@todonotes@disabled
 
 \else
 
 \fi
 \makeatother

\newtheorem{theorem}{Theorem}
\newtheorem{corollary}[theorem]{Corollary}

\newtheorem{definition}[theorem]{Definition}
\newtheorem{lemma}[theorem]{Lemma}
\newtheorem{claim}[theorem]{Claim}
\newtheorem{question}[theorem]{Question}
\newtheorem{remark}[theorem]{Remark}
\newtheorem{fact}[theorem]{Fact}

\title{How robustly can you predict the future?}

\author{Sean Cox}
\address{
Department of Mathematics and Applied Mathematics \\
Virginia Commonwealth University \\
1015 Floyd Avenue \\
Richmond, Virginia 23284, USA.  
}
\email{scox9@vcu.edu}

\author{Matthew Elpers}
\email{mattyelps@gmail.com}

\thanks{We thank the anonymous referee for their careful reading of the paper, their discovery (and fix) of an error in the original proof of Theorem \ref{thm_MainNegativeWeak}, and their improvement of part \ref{item_Cyclic} of Corollary \ref{cor_Pos}.  Partially funded by VCU SEED Grant of the first author.}



\begin{document}

\begin{abstract}
Hardin and Taylor~\cite{MR2384262} proved that any function on the reals---even a nowhere continuous one---can be correctly predicted, based solely on its past behavior, at almost every point in time.  They showed in \cite{MR3100500} that one could even arrange for the predictors to be robust with respect to simple time shifts, and asked whether they could be robust with respect to other, more complicated time distortions.  This question was partially answered by Bajpai and Velleman~\cite{MR3552748}, who provided upper and lower frontiers (in the subgroup lattice of $\text{Homeo}^+(\mathbb{R})$) on how robust a predictor can possibly be.  We improve both frontiers, some of which reduce ultimately to consequences of H\"older's Theorem (that every Archimedean group is abelian).
\end{abstract}

\maketitle


\input{Intro.tex}

\input{Holder.tex}

\input{PositiveResult.tex} 

\input{Corollaries.tex}

\input{NegativeResult.tex} 

\input{Conclusion.tex}

\end{document}

%% file: Intro.tex
\section{Introduction}\label{sec_Intro}

If $S$ is a nonempty set, $\boldsymbol{{}^{\mathbb{R}} S}$ will denote the set of total functions from $\mathbb{R}$ to $S$, and $\boldsymbol{{}^{\underset{\smile}{\mathbb{R}}} S}$ will denote the set of all $S$-valued functions $f$ such that $\text{dom}(f)=(-\infty,t_f)$ for some $t_f \in \mathbb{R}$.  An \textbf{$\boldsymbol{S}$-predictor} will refer to any function $\mathcal{P}$ with domain and codomain as follows:
\begin{equation}\label{eq_DomCodP}
\mathcal{P}: {}^{\underset{\smile}{\mathbb{R}}} S \to S.
\end{equation}
If $f:(-\infty,t_f) \to S$ is a member of ${}^{\underset{\smile}{\mathbb{R}}} S$, we could view $\mathcal{P}(f)$ as an attempt to predict which member of $S$ (which ``state") the function $f$ would/should/will pick out at ``time" $t_f$, \emph{if} $t_f$ were in its domain.  We gauge how well $\mathcal{P}$ makes predictions by asking:  for which $F \in {}^\mathbb{R} S$ and which $t \in \mathbb{R}$ does $\mathcal{P}$ correctly predict $F(t)$, based solely on $F|_{(-\infty,t)}$?  I.e., for which $F:\mathbb{R} \to S$ and $t \in \mathbb{R}$ does the equality
\begin{equation}\label{eq_CorrectPredictAtt}
\mathcal{P}\Big( F|_{(-\infty,t)} \Big) = F(t) \tag{*}
\end{equation}
hold?  

If we restricted our attention to \emph{continuous} $F:\mathbb{R} \to S$ (with respect to some nice topology on $S$) then the problem would trivialize, since $\mathcal{P}\Big( F|_{(-\infty,t)}  \Big)$ could simply pick out $\lim_{z \nearrow t} F(z)$, which depends only on $F|_{(-\infty, t)}$.  But if $|S|\ge 2$ then it is impossible to find a $\mathcal{P}$ such that the equation \eqref{eq_CorrectPredictAtt} holds for every function $F:\mathbb{R} \to S$ and every $t \in \mathbb{R}$.  Simply fix any $f: (-\infty,0) \to S$; since $|S|\ge 2$ there is a (possibly non-continuous) function $F:\mathbb{R} \to S$ such that $F|_{(-\infty,0)}=f$ and $F(0) \ne \mathcal{P}( f )$.  Then $\mathcal{P} \Big( F|_{(-\infty,0)}  \Big) = \mathcal{P}(f) \ne F(0)$, so $\mathcal{P}$ failed to predict $F(0)$.

Hardin and Taylor~\cite{MR2384262} considered what happens when we require the equality \eqref{eq_CorrectPredictAtt} to merely hold for \emph{almost} every $t$.  We will say that an $S$-predictor $\mathcal{P}$ is \textbf{good} if, for all total functions $F: \mathbb{R} \to S$, the equation \eqref{eq_CorrectPredictAtt} holds for all except measure-zero many $t \in \mathbb{R}$ (where the measure zero set of ``bad" predictions is allowed to depend on $F$).  They proved:
\begin{theorem}[Hardin-Taylor~\cite{MR2384262}]\label{thm_HardinTaylor}
For every set $S$, there exists a good $S$-predictor.
\end{theorem}
They showed in \cite{MR3100500} that the $S$-predictor could even be arranged to be independent of time shifts; i.e., to have the property that if $f$ is some horizontal (constant) shift of $g$, then $\mathcal{P}(f)=\mathcal{P}(g)$.  They point out that this result has its roots in a 1965 problem of Galvin about ``infinite hat" puzzles, which appeared in the problems section of the \emph{American Mathematical Monthly} (\cite{MR1533548}).  The problem elicited an incorrect solution (\cite{MR1534092}) and a later correction (\cite{MR1534397}), both of which could be viewed as precursors to the key ideas in the various Hardin-Taylor results, and to the following question:
\begin{question}[paraphrase of Question 7.8.3 of \cite{MR3100500}]\label{q_BigQ}
How robust can a good predictor be, with respect to distortions in time?
\end{question}

More precisely:  let $\text{Homeo}^+(\mathbb{R})$ denote the group (under composition) of increasing homeomorphisms of $\mathbb{R}$.  Following \cite{MR3100500}, if $U \subseteq \text{Homeo}^+(\mathbb{R})$,\footnote{We do not necessarily assume $U$ is a group; e.g., the set $C^\infty(\mathbb{R}) \cap \text{Homeo}^+(\mathbb{R})$ in Bajpai-Velleman's Theorem \ref{thm_BV_neg} below is not a group under composition, because it is not closed under inverses.} an $S$-predictor $\mathcal{P}$ is called \textbf{$\boldsymbol{U}$-anonymous}\footnote{``$U$-invariant" would also be an appropriate name, but we (following \cite{MR3552748}) reserve that terminology for certain functions from $\mathbb{R} \to S$; see Section \ref{sec_Pos}.} if $\mathcal{P}(f)=\mathcal{P}(f \circ \varphi)$ for all $\varphi \in U$ and all $f \in {}^{\underset{\smile}{\mathbb{R}}} S$.  Since $\text{dom}(f) = (-\infty,t_f)$ for some $t_f \in \mathbb{R}$, the domain of $f \circ \varphi$ is understood to be $\Big(-\infty, \varphi^{-1}(t_f)\Big)$.  See Figure \ref{fig_TopoSine} for an example (with $S=\mathbb{R}$).


\begin{figure}[h]
\caption{$f(x) = \sin(1/x)$ with domain $(-\infty,0)$, and $\varphi(x)=x+e^{x+5}$ is a particular member of $\text{Homeo}^+(\mathbb{R})$.  If $\varphi \in U \subseteq\text{Homeo}^+(\mathbb{R})$ and $\mathcal{P}:{}^{\protect\underset{\smile}{\mathbb{R}}} \mathbb{R} \to \mathbb{R}$ is a $U$-anonymous $\mathbb{R}$-predictor, then $\mathcal{P}(f)$ is required to be the same as $\mathcal{P}(f \circ \varphi)$.}
\label{fig_TopoSine}
\centering
\includegraphics[scale=.5]{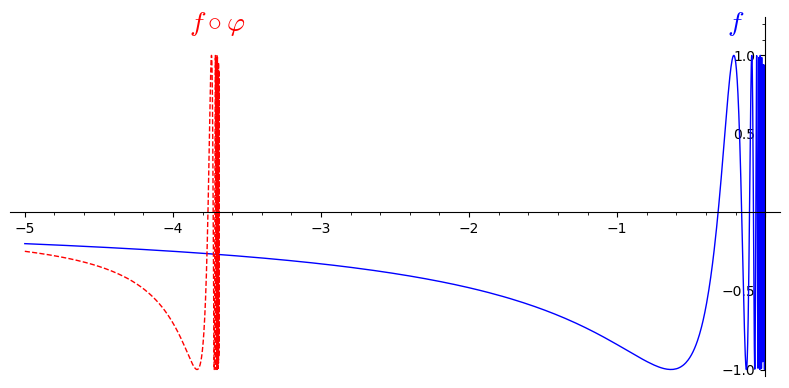}
\end{figure}

 So, $U$-anonymity of $\mathcal{P}$ means that $\mathcal{P}$ is insensitive to distortions in time caused by members of $U$.  The larger the set $U$, the more robust is the predictor $\mathcal{P}$.  We already mentioned that Hardin and Taylor produced good predictors that were independent of horizontal shifts; using the terminology above, their theorem can be rephrased as:  for every set $S$, there is a good $S$-predictor that is anonymous with respect to the group of shift functions (i.e., the group of functions of the form $x \mapsto x + c$ for some constant $c$).  This was improved by Bajpai and Velleman:
\begin{theorem}[Bajpai-Velleman~\cite{MR3552748}]\label{thm_BV_pos}
For any set $S$, there is a good $S$-predictor that is anonymous with respect to $\text{Aff}^+(\mathbb{R})$ (the group of affine functions of positive slope). 
\end{theorem}

The \emph{Axiom of Choice} was used in the proofs of both the Hardin-Taylor Theorem \ref{thm_HardinTaylor} and the Bajpai-Velleman Theorem \ref{thm_BV_pos}, and in our improvement to be discussed below (Theorem \ref{thm_MainPos}).  All of these results can be viewed as yet more strange consequences of the Axiom of Choice, the most famous of which are the Banach-Tarski Paradox and the existence of non-Lebesgue-measurable sets of reals.  In \cite{MR2384262}, Hardin and Taylor state:
\begin{quote}
``We should emphasize that these results do not give a practical means of predicting the future, just as the time dilation one would experience standing near the event horizon of a black hole does not give a practical time machine.  Nevertheless, we choose this presentation because we find it the most interesting, as well as pedagogically useful \dots"
\end{quote}

But the Axiom of Choice is not a panacea for obtaining anonymous predictors.  Even with the Axiom of Choice at their disposal, Bajpai and Velleman showed there is a limit to how robust predictors can be:

\begin{theorem}[Bajpai-Velleman~\cite{MR3552748}]\label{thm_BV_neg}
There is an equivalence relation $\sim$ on $\mathbb{R}$ such that, letting $S:=\mathbb{R}/\sim$, there is \textbf{no} good $S$-predictor that is anonymous with respect to the set
\[
C^\infty(\mathbb{R}) \cap \text{Homeo}^+(\mathbb{R}).
\] 
\end{theorem}

Clearly if $U_0 \subset  U_1$, then a $U_1$-anonymous predictor is also $U_0$-anonymous.  So Theorems \ref{thm_BV_pos} and \ref{thm_BV_neg} provide ``frontiers" in the subgroup lattice of $\text{Homeo}^+(\mathbb{R})$ for how anonymous we can require predictors to be:  the lower frontier is $\text{Aff}^+(\mathbb{R})$, and the upper frontier is $C^\infty(\mathbb{R}) \cap \text{Homeo}^+(\mathbb{R})$.  

Of course this is a large gap, and \cite{MR3552748} closed by asking what happens between those two extremes. Our main results address this problem, thus chipping away at Question \ref{q_BigQ}.  Our new Theorems \ref{thm_MainPos} and \ref{thm_MainNeg} below are strengthenings of Theorems \ref{thm_BV_pos} and \ref{thm_BV_neg}, respectively.    In order to more flexibly derive examples, we state Theorem \ref{thm_MainPos} in terms of $\text{Homeo}^+(I)$ where $I$ is any open interval of real numbers.  See Section \ref{sec_Pos} for the meaning of \emph{good $S$-predictor} in this context, which is the direct generalization of the definition given above (when $I=\mathbb{R}$).

\begin{theorem}\label{thm_MainPos}
Suppose $I$ is an open interval of real numbers, $U$ is a subgroup of $\text{Homeo}^+(I)$, and:
\begin{enumerate}
 \item\label{item_CommActFree} The commutator subgroup of $U$ acts freely on $I$;\footnote{A group $C \le \text{Homeo}^+(I)$ is said to act freely on $I$ if for every non-identity $\varphi \in C$, $\varphi$ has no fixed points.  This is a special case of the general notion of a free action of a group.} and 

 \item\label{item_AtMostOneFP} each non-identity member of $U$ has at most one fixed point.
\end{enumerate}
Then for any set $S$, there is a good $U$-anonymous $S$-predictor.
\end{theorem}

The proof of Theorem \ref{thm_MainPos} relies heavily on the ideas of the proofs in Bajpai-Velleman~\cite{MR3552748} and Hardin-Taylor~\cite{MR2384262}.  In fact, Theorem \ref{thm_MainPos} is essentially the result of noticing that the Hardin-Taylor theorem about shift-anonymous predictors was essentially due to the classic group-theoretic \emph{H\"older's Theorem} discussed in Section \ref{sec_Holder}, and that the Bajpai-Velleman strengthening was essentially about about fixed points and commutator subgroups.

\begin{corollary}\label{cor_Pos}
We list several examples that satisfy the assumptions of Theorem \ref{thm_MainPos}.

\begin{enumerate}[label=(\Alph*)]
 \item Bajpai-Velleman's Theorem \ref{thm_BV_pos} is an immediate consequence of Theorem \ref{thm_MainPos} (with $I=\mathbb{R}$ and $U=\text{Aff}^+(\mathbb{R})$), since the commutator subgroup of $\text{Aff}^+(\mathbb{R})$ is exactly the group of shift functions (which act freely on $\mathbb{R}$), and non-identity affine functions have at most one fixed point.

  \item\label{item_U_itself_free} When $U \le \text{Homeo}^+(I)$ and $U$ itself acts freely on $I$, then $U$ trivially satisfies both assumptions of Theorem \ref{thm_MainPos}.  For example, this holds when $I=(0,1)$ and $U$ is the group of $\varphi \in \text{Homeo}^+(I)$ of the form $x \mapsto x^\alpha$ for some $\alpha > 0$ (this $U$ acts freely on $(0,1)$).  

  \item\label{item_FiniteFP} If $U$ acts transitively on $I$,\footnote{I.e., for every $x,y \in I$ there is some $\varphi \in U$ such that $\varphi(x)=y$.} the commutator subgroup of $U$ acts freely on $I$, and each non-identity member of $U$ has only finitely many fixed points, then in fact each non-identity member of $U$ has at most one fixed point (see Lemma \ref{lem_ActTransFP} in Section \ref{sec_Cor} for a proof).  Hence, Theorem \ref{thm_MainPos} applies.

  \item\label{item_Example_analytic} (Special case of part \ref{item_FiniteFP}):  suppose $U$ is a group of analytic functions in $\text{Homeo}^+\Big( [0,1] \Big)$; then it is well-known that, since $[0,1]$ is compact, each non-identity member of $U$ has only finitely many fixed points.\footnote{This follows from the compactness of $[0,1]$ and the fact that the zeros of a nonzero analytic function form a discrete set.}  Note that members of $U$ also act on the set $(0,1)$. So, if $U$ acts transitively on $(0,1)$ and the commutator subgroup of $U$ acts freely on $(0,1)$, then we are in the setting of part \ref{item_FiniteFP}, with $I=(0,1)$ (and where we replace $U$ by $\{ \varphi \restriction (0,1) \ : \ \varphi \in U   \}$).
  
  \item\label{item_Cyclic} If $\varphi \in \text{Homeo}^+(I)$ and $\varphi$ has only measure zero many fixed points, then for every set $S$, there exists a $\langle \varphi \rangle$-invariant, good $S$-predictor (where $\langle \varphi \rangle$ denotes the cyclic group generated by $\varphi$).  See Section \ref{sec_Cor} for why this follows from Theorem \ref{thm_MainPos}.  In particular, this holds whenever $\varphi$ has only countably many fixed points (e.g., for any non-identity analytic function in $\text{Homeo}^+(I)$).  
  
\end{enumerate}
\end{corollary}

\begin{remark}
Since the assumptions of Theorem \ref{thm_MainPos} are phrased entirely in terms of continuous group actions, examples of $U$ with those properties transfer via homeomorphisms.  For example, if $I=(0,1)$ and 
\[
U = \{ x \mapsto x^\alpha \ : \ \alpha > 0 \}
\]
are as in part \ref{item_U_itself_free} of the corollary, and $\Psi: (0,1) \to \mathbb{R}$ is a homeomorphism, then 
  \[
\Psi U \Psi^{-1} =  \{ \Psi \circ \varphi \circ \Psi^{-1} \ : \ \varphi \in U   \}
  \]
is a subgroup of $\text{Homeo}^+(\mathbb{R})$ that satisfies the assumptions of Theorem \ref{thm_MainPos}.

\end{remark}

We also strengthen the ``upper frontier" from Bajpai-Velleman's Theorem \ref{thm_BV_neg}:
\begin{theorem}\label{thm_MainNeg}
There is an equivalence relation $\sim$ on $\mathbb{R}$ such that, letting $S:=\mathbb{R}/\sim$, there is \textbf{no} good $S$-predictor that is anonymous with respect to the group of infinitely Lipschitz diffeomorphisms.  (In fact we show there is no good ``weak" $S$-predictor that is anonymous with respect to this class; see Section \ref{sec_Neg} for the relevant definitions).
\end{theorem}

Section \ref{sec_Holder} provides the relevant background about H\"older's Theorem, Section \ref{sec_Pos} proves Theorem \ref{thm_MainPos}, Section \ref{sec_Cor} proves Corollary \ref{cor_Pos}, and Section \ref{sec_Neg} proves Theorem \ref{thm_MainNeg}.  \emph{The only prerequisites are very basic group theory and real analysis; in particular, no set-theoretic background whatsoever is assumed of the reader}.  The only set-theoretic ingredient is the use of a wellorder in the proof of Theorem \ref{thm_MainPos} in Section \ref{sec_Pos}.  A \textbf{wellorder} is a linear order with no infinite descending chain.  The Axiom of Choice ensures that every set can be wellordered, which we will use exactly once (on page \pageref{item_GroupII}) to obtain a wellorder of the set
\[
{}^I S := \{ F \ : \ F \text{ is a function from } I \text{ into } S \}.
\]

By \textbf{wellordered set of reals} we will mean a set $X \subset \mathbb{R}$ such that $(X,<)$ is a wellorder, where $<$ denotes the usual ordering of $\mathbb{R}$.  The following fact is key (see \cite[Corollary 3.4]{MR2384262}):
\begin{fact}\label{fact_WO_MeasureZero}
Every wellordered set of reals is countable (hence, has Lebesgue measure zero).
\end{fact}

%% file: Holder.tex
\section{H\"older's Theorem and its consequences}\label{sec_Holder}

An \textbf{Archimedean ordered group} is a triple $(G,\cdot, <)$ such that $(G,\cdot)$ is a group, $<$ is a linear order on $G$, and:
\begin{enumerate}
 \item  $<$ is both left and right invariant with respect to the group operation, meaning that for all $a,b,c \in G$:\footnote{The invariance requirements look more natural when one uses additive instead of multiplicative notation.  But in applications involving Archimedean orders one cannot usually know, a priori, that the group is abelian.}  
 \[
  a < b \implies ca < cb  \ \text{ and } \ ac < bc; 
  \]
  
 \item whenever $1 < a < b$, there is an $n \in \mathbb{N}$ such that $b < a^n$.
\end{enumerate}  
We also say that an (un-ordered) group $(G,\cdot)$ is \textbf{Archimedean} if there exists a linear order $<$ on $G$ such that $(G,\cdot,<)$ is an Archimedean ordered group.

\begin{theorem}[H\"older's Theorem; see Chapter IV, Theorem 1 of Fuchs~\cite{MR0171864}]\label{thm_Holder}
Every Archimedean group is abelian.\footnote{H\"older actually proved more; namely, that every Archimedean group is isomorphic to a subgroup of $(\mathbb{R},+)$.  We will not need this additional fact, however.}
\end{theorem} 

H\"older's Theorem is relevant for the study of continuous group actions on the reals:

\begin{theorem}[folklore; special case of Theorem 4 of \cite{MR1908556}]\label{thm_FolkAppHolder}
Suppose $I$ is an open interval of real numbers and $V$ is a subgroup of $\text{Homeo}^+(I)$ that acts freely on $I$.    Then the binary relation $\prec$ on $V$ defined by
\[
\varphi \prec \psi \ :\iff \ \ \forall x \in I \ \varphi(x) < \psi(x)
\]
is an Archimedean order on $V$.  Consequently, by H\"older's Theorem, $V$ is abelian.
\end{theorem}

Since Theorem \ref{thm_FolkAppHolder} is so central to the paper (in the key Lemma \ref{lem_KeyLemmaHoelder}), we give a sketch of the proof.  Fix any $x_0 \in I$, and define the relation $\prec_{x_0}$ on $V$ by:  $\varphi \prec_{x_0} \psi$ if $\varphi(x_0) < \psi(x_0)$.  The relation $\prec_{x_0}$ is clearly transitive and antisymmetric.  It is total, because if $\varphi,\psi \in V$ and $\varphi(x_0) = \psi(x_0)$ then $\psi^{-1} \circ \varphi$ fixes $x_0$, and hence (since $V$ acts freely on $I$) $\psi^{-1} \circ \varphi = \text{id}$; i.e., $\varphi=\psi$.  So $\prec_{x_0}$ is a linear order on $V$.  The Intermediate Value Theorem, together with the assumption that $V$ acts freely on $I$, ensure that $\prec_{x_0}$ is the same as the relation $\prec$ defined in the statement of Theorem \ref{thm_FolkAppHolder}. This relation is also left and right invariant with respect to composition (on $V$).  Finally, to see that $\prec_{x_0}$ is Archimedean, suppose $\text{id} \prec_{x_0} \varphi \prec_{x_0} \psi$ with $\varphi,\psi \in V$.  Then $\text{id} \prec \varphi$, and hence $\langle \varphi^n(x_0) \ : \ n \in \mathbb{N} \rangle$ is a strictly increasing sequence.  So it either converges to a point in $I$, or to ``$+\infty$" (the right endpoint of $I$).  It cannot converge to a point in $I$, because if it did, that limit would be a fixed point of the continuous (non-identity) function $\varphi$, contradicting that $\varphi  \in V$ and $V$ acts freely.  So $\lim_{n \to \infty} \varphi^n(x_0) = \infty$; hence, there is some $n$ such that $\psi(x_0) < \varphi^n(x_0)$.  Then $\psi \prec_{x_0} \varphi^n$, which is equivalent to saying $\psi \prec \varphi^n$.

%% file: PositiveResult.tex
\section{Proof of Theorem \ref{thm_MainPos}}\label{sec_Pos}

In this section we prove Theorem \ref{thm_MainPos}.  If $I$ is an open interval of real numbers, the lower and upper endpoints of $I$ will be denoted $-\infty$ and $\infty$, respectively.  If $S$ is a nonempty set, $\boldsymbol{{}^I S}$ denotes the set of total functions from $I$ to $S$, and $\boldsymbol{{}^{\underset{\smile}{I}} S}$ denotes the set of functions of the form $f: (-\infty,t_f) \to S$ for some $t_f \in I$.  Uppercase and lowercase letters will be used for members of ${}^I S$ and ${}^{\underset{\smile}{I}} S$, respectively, but note that if $F \in {}^I S$ and $t \in I$, then $F|_{(-\infty,t)}$ is a member of ${}^{\underset{\smile}{I}} S$.

The notions of good and anonymous $S$-predictors in this context are the obvious generalizations of the definitions given in the introduction (which were for the special case $I=\mathbb{R}$); i.e, a good $S$-predictor is a function $\mathcal{P}: {}^{\underset{\smile}{I}} S \to S$ such that for all $F \in {}^I S$, 
\[
\left\{ t \in I \ : \ \mathcal{P}\Big( F|_{(-\infty,t)} \Big) \ne F(t) \right\} \text{ has Lebesgue measure zero.} 
\]
If $U \subseteq \text{Homeo}^+(I)$, we say that a function $\mathcal{P}$ with domain ${}^{\underset{\smile}{I}} S$ is $U$-anonymous if $\mathcal{P}(f) = \mathcal{P}(f \circ \varphi)$ for every $f \in {}^{\underset{\smile}{I}} S$ and every $\varphi \in U$.

Lowercase Greek letters are reserved for elements of $\text{Homeo}^+(I)$.  If $\varphi \in \text{Homeo}^+(I)$ and $n \in \mathbb{Z}$, $\varphi^n$ denotes the $n$-fold composition of $\varphi$ if $n \ge 1$, $\text{id}_I$ if $n=0$, and the $|n|$-fold composition of $\varphi^{-1}$ if $n < 0$.  

 If $F:I \to S$ and $\varphi \in \text{Homeo}(I)$, we say that \textbf{$\boldsymbol{F}$ is $\boldsymbol{\varphi}$-invariant} if $F = F \circ \varphi$, and \textbf{$\boldsymbol{F}$ is $\boldsymbol{\varphi}$-invariant before $\boldsymbol{t}$} if 
\[
F \restriction (-\infty,t) = F \circ \varphi \restriction (-\infty,t).
\]
We say that \textbf{$\boldsymbol{F}$ is past $\boldsymbol{\varphi}$-invariant} if there is some $t \in I$ such that $F$ is $\varphi$-invariant before $t$.

\begin{lemma}[analogue of Lemmas 2 and 3 of \cite{MR3552748}]\label{lem_EasyExtend}
Suppose $F \in {}^I S$, $\varphi \in \text{Homeo}^+(I)$, and $F$ is $\varphi$-invariant before $t$.  Then there is an $H:I \to S$ that extends $F|_{(-\infty,t)}$ and is $\varphi$-invariant.

\end{lemma}
\begin{proof}
Since $\varphi$ is invertible, the relation 
		\[ x\sim z  \ :\iff  \ \exists  n \in \mathbb{Z} \  \varphi^n(x)=z\]
is an equivalence relation on $I$.  Suppose $x \sim z$; without loss of generality say $n \ge 0$ and $\varphi^n(x) = z$.  Then, because $\varphi$ is order preserving, the sequence $\Big( \varphi^k(x) \Big)_{k \ge 0}$ is monotonic (increasing if $x \le \varphi(x)$, decreasing if $x \ge \varphi(x)$).  In particular, if both $x$ and $z$ are less than $t$, then so is $\varphi^k(x)$ for every integer $k$ between $0$ and $n$.  Since $F$ is $\varphi$-invariant before $t$, it follows that $F(x) = F(\varphi^k(x))$ for each $k$ between $0$ and $n$; in particular, $F(x)=F(z)$.  

Hence, for any equivalence class $C$, if $C \cap (-\infty,t) \ne \emptyset$, then the restriction of $F$ to $C \cap (-\infty,t)$ is constant; let $s_C$ denote this constant value.  Now fix any $s^* \in S$, and define $H:I \to S$ by
\[
H(y)=
\begin{cases}
s_{[y]} & \text{ if } [y] \cap (-\infty,t) \ne \emptyset \\
s^* & \text{ otherwise,}
\end{cases}
\]
where $[y]$ denotes the equivalence class of $y$.  Then $H$ extends $F|_{(-\infty,t)}$ and, since $[y]=[\varphi(y)]$ for every $y \in I$, $H$ is $\varphi$-invariant.
\end{proof}

The next lemma is the key use of free actions and H\"older's Theorem:
\begin{lemma}[analogue of Lemma 4 of \cite{MR3552748}]\label{lem_KeyLemmaHoelder}
Suppose $V$ is a subgroup of $\text{Homeo}^+(I)$ and $V$ acts freely on $I$.  Suppose $F: \mathbb{R} \to S$ is invariant with respect to some non-identity member of $V$.  Then for any $\varphi \in V$:  if $F$ is past $\varphi$-invariant, then $F$ is (fully) $\varphi$-invariant.
\end{lemma}
\begin{proof}
By assumption there is a $\psi \in V$ such that $\psi \ne \text{id}$ and $F$ is $\psi$-invariant.  Now suppose $\varphi \in V$ and $F$ is past $\varphi$-invariant; recall this means there is some $t \in I$ such that $F$ is $\varphi$-invariant before $t$ (i.e, if $z < t$ then $F(\varphi(z)) = F(z)$).  Fix any $x \in I$; we want to show $F(\varphi(x)) = F(x)$.  Since $\psi \ne \text{id}$ and $V$ acts freely, $\psi$ has no fixed points.  Let $\prec$ be the linear order on $V$ given by Theorem \ref{thm_FolkAppHolder}.  Then either $\psi \prec \text{id}$ or $\psi^{-1} \prec \text{id}$.  Without loss of generality\footnote{By replacing $\psi$ with $\psi^{-1}$ if necessary; note that if $F$ is $\psi$-invariant then it is also $\psi^{-1}$-invariant.} we can assume that $\psi \prec \text{id}$.  Then $\langle \psi^n(x) \ : \ n \in \mathbb{N} \rangle$ is a strictly decreasing sequence.  Since $\psi$ is continuous and has no fixed points, it follows that $\lim_{n \to \infty} \psi^n(x) = -\infty$.  So for some $n \in \mathbb{N}$, $\psi^n(x) <  t$.  Then
\begin{align*}
F(x) & = F\big( \psi^n(x) \big) & \text{ (because $F$ is $\psi$-invariant)} \\
& = F\big( \varphi(\psi^n(x)) \big) & \text{ (because $\psi^n(x) < t$ and $F$ is $\varphi$-invariant before $t$ )} \\
& = F\big( \psi^n(\varphi(x)) \big)  & \text{ (because $V$ is abelian, by Theorem \ref{thm_FolkAppHolder})}  \\
& = F(\varphi(x)) & \text{ (because $F$ is $\psi$-invariant).}
\end{align*}
\end{proof}

A \textbf{commutator} in a group is an element of the form $aba^{-1} b^{-1}$.   If $U$ is a group, $\boldsymbol{[U,U]}$ denotes the \textbf{commutator subgroup of $U$}, which is the subgroup generated by the commutators.  At one point we will use the basic group-theoretic fact that
\begin{equation*}
[U,U] \text{ is a normal subgroup of } U.
\end{equation*}

We now commence with the proof of Theorem \ref{thm_MainPos}.  Fix any nonempty set $S$ for the remainder of this section.   Assume $I$ is an open interval of real numbers and $U$ is a subgroup of $\text{Homeo}^+(I)$ such that
\begin{enumerate}
 \item $[U,U]$ acts freely on $I$, and
 \item each non-identity member of $U$ has at most one fixed point.
\end{enumerate}

Define the following subsets of ${}^I S$:
\begin{itemize} \item\label{item_GroupI} $\text{Inv}_{[U,U]}:=$ the set of functions in ${}^I S$ that are invariant with respect to some non-identity member of $[U,U]$.
 \item\label{item_GroupII} $\text{Inv}_{U}:=$ the set of functions in ${}^I S$ that are invariant with respect to some non-identity member of $U$.
\end{itemize}

By the Axiom of Choice, there exists a well ordering $\triangleleft$ of ${}^I S$ such that $\text{Inv}_{[U,U]}$ functions are listed first, then functions in $\text{Inv}_U \setminus \text{Inv}_{[U,U]}$, then the rest of ${}^I S$.\footnote{The Axiom of Choice directly provides wellorders $\triangleleft_0$, $\triangleleft_1$, and $\triangleleft_2$ for the (pairwise disjoint) sets $\text{Inv}_{[U,U]}$, $\text{Inv}_U \setminus \text{Inv}_{[U,U]}$, and ${}^I S \setminus \text{Inv}_U$, respectively.  Stacking these three wellorders yields the desired wellorder $\triangleleft$ of ${}^I S$.}  The reason for these requirements will become apparent in the proof of Claim \ref{clm_Extensions}.

Recall from the discussion of anonymity in Section \ref{sec_Intro} that if $f \in {}^{\underset{\smile}{I}} S$ and $\varphi \in \text{Homeo}^+(I)$, $f \circ \varphi$ is understood to have domain $\big(-\infty, \varphi^{-1}(t_f) \big)$.  Define $\boldsymbol{ f \circ U}$ to be the set of functions of the form $f \circ \varphi$, where $\varphi \in U$.  Let $\boldsymbol{\overline{f \circ U}}$ denote the collection of all (total) $F: I \to S$ such that $F$ extends at least one member of $f \circ U$. Note that a given $F \in \overline{f \circ U}$ may possibly extend more than one member of $f \circ U$;\footnote{For example (in the case dealt with by Bajpai-Velleman):  if $F$ happens to extend $f$ and $F$ is periodic with period $p$, then $F$ also extends $f \circ \varphi_k$ where $\varphi_k(x)=x+kp$ for any $k \in \mathbb{N}$.  }  in fact, dealing with this issue is the heart of the argument.  Let $F_{f \circ U}$ denote the $\triangleleft$-least member of the set $\overline{f \circ U}$.  

Part \eqref{item_StrongInv} of the following claim is the key use of Lemma \ref{lem_KeyLemmaHoelder}, which was a consequence of H\"older's Theorem.
\begin{claim}\label{clm_Extensions}
Suppose $f:(-\infty,t_f) \to S$, $\varphi \in U$, and $F_{f \circ U}$ extends $f \circ \varphi$.  Suppose $\gamma \in U$, $\gamma \ne \text{id}_I$, and $F_{f \circ U}$ is $\gamma$-invariant before $\varphi^{-1}(t_f)$ (i.e., for inputs that come from the domain of $f \circ \varphi$).  Then:
\begin{enumerate}
 \item\label{item_WeakInv} $F_{f \circ U} \in \text{Inv}_{U}$ (i.e., $F_{f \circ U}$ is fully invariant with respect to some non-identity member of $U$, though not necessarily with respect to $\gamma$).
 \item\label{item_StrongInv} If $\gamma \in [U,U]$, then $F_{f \circ U}$ is fully $\gamma$-invariant. 
\end{enumerate} 
\end{claim}
\begin{proof}
By Lemma \ref{lem_EasyExtend}, there is an $H:\mathbb{R} \to S$ that extends $F_{f \circ U}|_{\big(-\infty, \varphi^{-1}(t_f) \big)}$ ($=f \circ \varphi$) and is $\gamma$-invariant.  Then $H \in \overline{f \circ U} \cap  \text{Inv}_U$.  Since $F_{f \circ U}$ is the $\triangleleft$-least member of $\overline{f \circ U}$, $F_{f \circ U} \trianglelefteq H$.  And since $H \in \text{Inv}_U$ and $\triangleleft$ lists members of $\text{Inv}_U$ before members of ${}^I S \setminus \text{Inv}_U$, it follows that $F_{f \circ U} \in \text{Inv}_U$.  This proves part \eqref{item_WeakInv}.

If $\gamma$ was an element of $[U,U]$, then $H \in \overline{f \circ U} \cap \text{Inv}_{[U,U]}$, and since $\triangleleft$ lists members of $\text{Inv}_{[U,U]}$ before members of $\text{Inv}_U \setminus \text{Inv}_{[U,U]}$, it follows that $F_{f \circ U} \in \text{Inv}_{[U,U]}$.  So $F_{f \circ U}$ is invariant with respect to some non-identity member of $[U,U]$.  By Lemma \ref{lem_KeyLemmaHoelder} and our assumption that $[U,U]$ acts freely on $I$, $F_{f \circ U}$ is fully $\gamma$-invariant.
\end{proof}

Define
\[
\mathcal{P}: {}^{\underset{\smile}{I}} S \to S
\]
by
\begin{equation}\label{eq_DefP}
\mathcal{P}(f):= F_{f \circ U}(\underbrace{\varphi^{-1}(t_f)}_{\substack{\text{Right end-} \\ \text{point of} \\ \text{dom}\big( f \circ \varphi \big) }}),
\end{equation}
where $\varphi$ is \emph{any} member of $U$ witnessing that $F_{f \circ U}$ is an element of $\overline{f \circ U}$; i.e., where $\varphi$ is  \emph{any} member of $U$ such that $F_{f \circ U}$ extends $f \circ \varphi$.  The next claim says that $\mathcal{P}$ is well-defined, in the sense that the expression above does not depend on the choice of the $\varphi$:

\begin{claim}\label{claim_WellDef}
The definition of $\mathcal{P}(f)$ in \eqref{eq_DefP} does not depend which $\varphi \in U$ we choose to witness $F_{f \circ U} \in \overline{f \circ U}$.  I.e., if $\varphi_1$ and $\varphi_2$ are both in $U$, and $F_{f \circ U}$ extends both $f \circ \varphi_1$ and $f \circ \varphi_2$, then 
\[
F_{f \circ U}\big( \varphi_1^{-1}(t_f) \big) = F_{f \circ U}\big( \varphi_2^{-1}(t_f) \big).
\]
\end{claim}
\begin{proof}
We follow the proof in \cite[Theorem 5]{MR3552748} very closely (which was for the special case where $I=\mathbb{R}$ and $U$ was the group of increasing affine functions).  

For the remainder of the proof of the claim, we abbreviate $F_{f \circ U}$ by $F$, and omit the ``$\circ$" symbol in compositions.  Let $x_1=\varphi^{-1}_1(t_f)$ and $x_2=\varphi^{-1}_2(t_f)$; so the domain of $f \varphi_1$ is $(-\infty,x_1)$, the domain of $f  \varphi_2$ is $(-\infty,x_2)$, and $F$ extends both of them.  We must show that $F(x_1)=F(x_2)$.  Clearly this holds if $x_1=x_2$, so suppose from now on that $x_1 \ne x_2$.  In particular, 
\begin{equation}
\varphi^{-1}_1  \varphi_2 \ \text{ is a non-identity member of } U
\end{equation}
because it moves $x_2$ to $x_1$.  Since $F$ extends $f \varphi_1$ before $x_1$ and extends $f \varphi_2$ before $x_2$, it follows that:\footnote{To see $\varphi^{-1}_1 \varphi_2$-invariance before $x_2$, suppose $z < x_2 = \varphi_2^{-1}(t_f)$. Since $\varphi_2(z)  < \varphi_2(x_2) = t_f$, applying the order-preserving $\varphi_1^{-1}$ yields that $\varphi_1^{-1} \varphi_2(z) < \varphi_1^{-1}(t_f) = x_1$.  Since $F$ extends $f \circ \varphi_1$ before $x_1$, $F \varphi_1^{-1} \varphi_2(z) = f \varphi_1 \varphi_1^{-1} \varphi_2(z)$, but the latter is just $f \varphi_2(z)$, which (since $z < x_2$ and $F$ extends $f \varphi_2$ before $x_2$) is the same as $F(z)$.  The proof of $\varphi^{-1}_2 \circ \varphi_1$-invariance of $F$ before $x_1$ is similar.}
\begin{equation}\label{eq_Past_F_Inv}
F \text{ is } \varphi^{-1}_1  \varphi_2 \text{-invariant before } x_2\text{, and } \varphi^{-1}_2  \varphi_1 \text{-invariant before } x_1.
\end{equation}

Since $F=F_{f \circ U}$ extends $f \circ \varphi_2$, $F$ is $\varphi_1^{-1} \varphi_2$-invariant before $\varphi_2^{-1}(t_f)=x_2$, and both $\varphi_2$ and $\varphi_1^{-1} \varphi_2$ are in $U$, part \ref{item_WeakInv} of Claim \ref{clm_Extensions} implies
\begin{equation}
F \in \text{Inv}_U.
\end{equation}
So there is some $\psi \in U$ such that $\text{id}_I \ne \psi$ and $F = F \psi$.  Without loss of generality---since both $\psi$ and $\psi^{-1}$ are both order preserving and $F$ is invariant with respect to both---we can assume that 
\begin{equation}\label{eq_Psix_1}
\psi(x_1) \le x_1.
\end{equation}

Still following the proof of \cite[Theorem 5]{MR3552748}, we consider two cases.  The assumption that non-identity members of $U$ have at most one fixed point is used in Case 1, and the assumption that $[U,U]$ acts freely on $I$ is used in Case 2.

\textbf{Case 1: $\psi$ commutes with $\varphi_2^{-1} \varphi_1$.}  We first claim that $\psi(x_1) < x_1$.  If not, then by \eqref{eq_Psix_1}, $x_1$ is a fixed point of $\psi$.  Then

\[
\psi(\underbrace{x_2}_{\varphi_2^{-1} \varphi_1(x_1)}) = \psi \varphi_2^{-1} \varphi_1 (x_1) \overset{\text{\tiny by case}}{=} \varphi_2^{-1} \varphi_1 \underbrace{\psi(x_1)}_{x_1} = x_2,
\] 
so $x_2$ is another fixed point of $\psi$.  This contradicts that non-identity members of $U$ have at most one fixed point.  

So $\psi(x_1) < x_1$.  Then

\begin{align*}
F (x_1) &  = F \psi (x_1) & \text{ (because $F$ is $\psi$-invariant)} \\
& = F \varphi_2^{-1} \varphi_1 \psi(x_1) & \text{ (because $\psi(x_1) < x_1$ and \eqref{eq_Past_F_Inv} )} \\
& = F \psi\varphi_2^{-1} \varphi_1(x_1) & \text{ ($\psi$ commutes with $\varphi_2^{-1} \varphi_1$ by the Case)} \\
& = F \psi(x_2)  & \\
& = F(x_2). & \text{ (because $F$ is $\psi$-invariant)} 
\end{align*}

\textbf{Case 2: $\psi$ does not commute with $\varphi_2^{-1} \varphi_1$}.  Let $\tau:= \varphi_2^{-1} \varphi_1$, and consider the commutator $\psi^{-1} \tau  \psi \tau^{-1}$.  By our case, 
\begin{equation}\label{eq_BetaNonIdCommutator}
\beta:=\psi^{-1} \tau  \psi \tau^{-1} \text{ is a non-identity member of } [U,U].
\end{equation}

First we show that
\begin{equation}\label{eq_betaInvBefore}
F \text{ is } \beta \text{-invariant before } x_2.
\end{equation}
  To see this, first observe that $\tau(x_1) = x_2$.  Suppose $z < x_2$.  Then $\tau^{-1} (z) < x_1$, and by \eqref{eq_Psix_1}, $\psi \tau^{-1}(z) < x_1$, and hence (since $\tau(x_1) = x_2$) $\tau \psi \tau^{-1}(z) < x_2$.  Then
\begin{align*}
F \beta (z) & = F \psi^{-1} \tau \psi \tau^{-1}(z) \\
& = F \tau \psi \tau^{-1}(z) & \text{($F$ is $\psi$-invariant)} \\
 & = F \underbrace{\varphi_1^{-1} \varphi_2}_{\tau^{-1}} \tau \psi \tau^{-1}(z) & \text{(by \eqref{eq_Past_F_Inv} and since $\tau \psi \tau^{-1}(z) < x_2$)} \\
  & = F \psi \tau^{-1}(z) & \\
  & = F \tau^{-1}(z) & \text{($F$ is $\psi$-invariant)} \\
  & = F \varphi_1^{-1} \varphi_2 (z) &   \\
 &  = F(z) &  \text{(by \eqref{eq_Past_F_Inv} and $z < x_2$)},
\end{align*}
which concludes the proof of \eqref{eq_betaInvBefore}.  By \eqref{eq_BetaNonIdCommutator}, \eqref{eq_betaInvBefore}, and part \ref{item_StrongInv} of Claim \ref{clm_Extensions}, 
\begin{equation}
F \text{ is (fully) } \beta \text{-invariant.}
\end{equation}

Since $\text{id} \ne \beta \in [U,U]$ and $[U,U]$ acts freely on $I$, in particular $x_2$ is not a fixed point of $\beta$.  Let $\alpha$ denote whichever one of $\beta$, $\beta^{-1}$ sends $x_2$ to an output strictly less than $x_2$.  Then:
\begin{equation}\label{eq_PropertiesAlpha}
\alpha \in [U,U], \ \alpha(x_2) <  x_2, \text{ and } F \text{ is } \alpha \text{-invariant.}
\end{equation}

Next we claim that
\begin{equation}\label{eq_Tauinalphatau_inv}
F \text{ is } \tau^{-1} \alpha \tau \text{-invariant before } x_1.
\end{equation}
Suppose $z < x_1$.  Then
\begin{align*}
F \tau^{-1} \alpha \tau(z) & = F \varphi_1^{-1} \varphi_2 \underbrace{\alpha \overbrace{\varphi_2^{-1} \varphi_1 (z)}^{<x_2}}_{\substack{<x_2 \text{ because } \\ \alpha(x_2) \le x_2 }} & \\ 
& = F \alpha \varphi_2^{-1} \varphi_1(z) &  \text{ by \eqref{eq_Past_F_Inv}}  \\
 & =  F \varphi_2^{-1} \varphi_1(z) & \text{($F$ is $\alpha$-invariant)} \\
 & = F(z) & \text{by \eqref{eq_Past_F_Inv} and $z < x_1$},
\end{align*}
which concludes the proof of \eqref{eq_Tauinalphatau_inv}.

Recall that commutator subgroups are always normal; hence, $\tau^{-1} \alpha \tau$ is an element of $[U,U]$.  Then by \eqref{eq_Tauinalphatau_inv}, the fact that $F$ extends $f \varphi_1$, and since $\varphi_1 \in U$ and $\tau^{-1} \alpha \tau \in [U,U]$, part \ref{item_StrongInv} of Claim \ref{clm_Extensions} ensures that
\begin{equation}\label{eq_FullyConjugateInv}
F \text{ is (fully) } \tau^{-1} \alpha \tau \text{-invariant.}
\end{equation}

Finally,
\begin{align*}
F(x_1) & =F \tau^{-1} \underbrace{\alpha \overbrace{\tau (x_1)}^{x_2}}_{<x_2 \text{ by } \eqref{eq_PropertiesAlpha}} & \text{by \eqref{eq_FullyConjugateInv}} \\
& = F \alpha \tau (x_1)& \text{ by \eqref{eq_Past_F_Inv} } \\
& = F \tau(x_1)  & \text{ $F$ is $\alpha$-invariant} \\
& = F(x_2). &
\end{align*}
\end{proof}

With Claim \ref{claim_WellDef} in hand, we can now finish the proof of Theorem \ref{thm_MainPos}.  To see the $U$-anonymity of $\mathcal{P}$, suppose  $f,g \in {}^{\underset{\smile}{I}} S$, $\tau \in U$, $\tau(t_g) = t_f$, and $g = f \circ \tau$.  Then $g \in f \circ U$, and it follows that $f \circ U = g \circ U$ and hence
\begin{equation}
F_{f \circ U} = F_{g \circ U}.
\end{equation}
Let $F$ denote this common function.  Let $\varphi \in U$ witness that $F \in \overline{g \circ U}$; i.e., $F$ extends $g \circ \varphi= f \circ \tau \circ \varphi$.  So $F = F_{f \circ U}$ extends $f \circ (\tau \circ \varphi)$, and since $\tau \circ \varphi \in U$, Claim \ref{claim_WellDef} ensures that $\mathcal{P}(f) = F\Big( (\tau \circ \varphi)^{-1}(t_f) \Big)$.  Then
\begin{align}
\mathcal{P}(g)  = F \big( \varphi^{-1}(\underbrace{t_g}_{\tau^{-1}(t_f)}) \big) = F\Big( (\tau \circ \varphi)^{-1} (t_f) \Big)= \mathcal{P}(f).
\end{align}

To see that $\mathcal{P}$ is good (this argument first appeared in Hardin-Taylor~\cite{MR2384262}):  suppose $H: I \to S$.  Let
\[
B:= \left\{ t \in \mathbb{R} \ : \ \mathcal{P}\left( H|_{(-\infty,t)} \right) \ne H(t) \right\}
\]
By Fact \ref{fact_WO_MeasureZero} (from page \pageref{fact_WO_MeasureZero}), to see that $B$ has measure zero---in fact, is countable and nowhere dense---it suffices to show that $B$ is a wellordered subset of $\mathbb{R}$ under the usual ordering $<$ on $\mathbb{R}$; and for that it suffices find an order-preserving embedding from $(B,<)$ into $\left( {}^I S, \triangleleft \right)$.\footnote{Since any descending chain from $(B,<)$ would get carried by the order-preserving embedding to a descending chain in $\left( {}^I S, \triangleleft \right)$, and hence would have to terminate at some finite stage because $\triangleleft$ is a wellorder.}  Define 
\[
e: (B,<) \to \left( {}^I S, \triangleleft \right), \ \ t \mapsto F_{ H|_{(-\infty,t)} \circ U}.
\]
To see that $e$ is order-preserving, suppose $s < t$ are both in $B$; we will show that $e(s) \triangleleft e(t)$ by showing $e(s) \trianglelefteq e(t)$ and $e(s) \ne e(t)$.  We first prove the non-strict inequality, i.e., that that
  \[
 F_{ H|_{(-\infty,s)} \circ U} \trianglelefteq F_{ H|_{(-\infty,t)} \circ U}.
 \]
 Since $F_{ H|_{(-\infty,s)} \circ U}$ is, by definition, the $\triangleleft$-least member of $\overline{H|_{(-\infty,s)} \circ U}
$, it suffices to show that 
 \begin{equation}\label{eq_F_H_in}
 F_{ H|_{(-\infty,t)} \circ U} \ \in \ \overline{H|_{(-\infty,s)} \circ U}.
 \end{equation}
Now $F_{ H|_{(-\infty,t)} \circ U}$, being a member of $\overline{H|_{(-\infty,t)} \circ U}$, extends $H|_{(-\infty,t)} \circ \varphi$ for some $\varphi \in U$.   But since $s < t$ (and $\varphi$ is order preserving), it also extends $H|_{(-\infty,s)} \circ \varphi$.  This verifies \eqref{eq_F_H_in}.
 
To prove that 
\[
 F_{ H|_{(-\infty,s)} \circ U} \ne F_{ H|_{(-\infty,t)} \circ U},
\]
suppose toward a contradiction that they are equal; let $F$ denote the common function.  Let $\varphi \in U$ witness that $F \in \overline{H|_{(-\infty,t)} \circ U}$; so $F$ extends $H|_{(-\infty,t)} \circ \varphi$.  Since $s < t$ and $\varphi$ is order preserving, $F$ also extends $H|_{(-\infty,s)} \circ \varphi$.  Then by Claim \ref{claim_WellDef} and the fact that $F = F_{H|_{(-\infty,s)} \circ U}$, 
\begin{equation}\label{eq_MainEqPF}
\mathcal{P}\Big(H|_{(-\infty,s)} \Big) = F\Big( \varphi^{-1}(s) \Big) 
\end{equation}
Now the left side of \eqref{eq_MainEqPF} is \textbf{not} equal to $H(s)$, because $s \in B$.  So 
\begin{equation}\label{eq_NE_to_contradict}
F\Big( \varphi^{-1}(s) \Big) \ne H(s).
\end{equation}
  On the other hand, $F$ extends $H|_{(-\infty,t)}  \circ \varphi$, whose domain is $(-\infty, \varphi^{-1}(t))$; and since $s < t$ and $\varphi$ (and hence $\varphi^{-1}$) is order preserving, $\varphi^{-1}(s) < \varphi^{-1}(t)$.  So $F$ and $H|_{(-\infty,t)}  \circ \varphi$ agree on the input $\varphi^{-1}(s)$.  Then $F\big(\varphi^{-1}(s)\big) = H|_{(-\infty,t)} \circ \varphi \big( \varphi^{-1}(s) \big) = H\big( \varphi(\varphi^{-1}(s)) \big) = H(s)$, contradicting \eqref{eq_NE_to_contradict}.  
  
This concludes the proof of Theorem \ref{thm_MainPos}.

%% file: Corollaries.tex
\section{Proof of Corollary \ref{cor_Pos}}\label{sec_Cor}

In this section we prove parts \ref{item_FiniteFP} and \ref{item_Cyclic} of Corollary \ref{cor_Pos}.  The remaining parts should be clear.  The following lemma proves part \ref{item_FiniteFP}.

\begin{lemma}\label{lem_ActTransFP}
Suppose $U$ is a subgroup of $\text{Homeo}^+(I)$, $U$ acts transitively on $I$, $[U,U]$ acts freely on $I$, and each non-identity member of $U$ has only finitely many fixed points.  Then in fact each non-identity member of $U$ has at most one fixed point.
\end{lemma}
\begin{proof}
Suppose toward a contradiction that there is some non-identity $\varphi \in I$ and some $x_1 < x_2$ in $I$ such that both $x_1$ and $x_2$ are fixed points of $\varphi$.  Since $U$ acts transitively on $I$, there is some $\tau \in U$ such that $\tau(x_1) = x_2$.  Then $x_2$ is a fixed point of $\tau \varphi \tau^{-1} \varphi^{-1}$.  We claim that $\tau \varphi \tau^{-1} \varphi^{-1} \ne \text{id}_I$, which will contradict the assumption that $[U,U]$ acts freely on $I$.

Suppose toward a contradiction that $\tau \varphi \tau^{-1} \varphi^{-1} = \text{id}_I$; i.e., that 
\begin{equation}\label{eq_ConjugItslf}
\tau \varphi \tau^{-1}=\varphi.
\end{equation}
  For $n \ge 3$ set $x_n:= \tau(x_{n-1})$.  Since $\tau$ is increasing and $\tau(x_1) = x_2$, the sequence $(x_n)_n$ is a strictly increasing sequence in $I$.  We prove by induction that each $x_n$ is a fixed point of $\varphi$, which will contradict that $\varphi$ has only finitely many fixed points.  That $x_1$ and $x_2$ are fixed points is by assumption.  Now if $x_n$ is a fixed point of $\varphi$, then
\begin{align*}
\varphi(x_{n+1}) & = \tau \varphi \tau^{-1} (x_{n+1}) & \text{by \eqref{eq_ConjugItslf}} \\ 
&= \tau \varphi \tau^{-1} \big( \tau(x_n) \big) & \text{definition of $x_{n+1}$ } \\
& = \tau \varphi(x_n) & \\
& = \tau(x_n) & \text{Induction Hypothesis} \\  
& = x_{n+1}. &
\end{align*}

\end{proof}

The ``at most one" in the conclusion of Lemma \ref{lem_ActTransFP} is best possible, since when $I=\mathbb{R}$ and $U$ is the group of increasing affine functions, the commutator subgroup $[U,U]$ is exactly the group of shift functions (i.e., affine functions of slope 1).  And in that case, $U$ acts transitively on $\mathbb{R}$, $[U,U]$ acts freely on $\mathbb{R}$, and members of $U \setminus [U,U]$ have exactly one fixed point.

Next we prove part \ref{item_Cyclic} of Corollary \ref{cor_Pos}.   Suppose $\varphi \in \text{Homeo}^+(I)$, and that the set $C$ of fixed points of $\varphi$ has Lebesgue measure zero.  Since $\varphi$ is continuous, $C$ is closed.  For $t \in I \setminus C$, let $a(t)$ denote the largest member of $C \cap (-\infty,t)$ if that intersection is nonempty; otherwise set $a(t)=-\infty$.  Closure of $C$, together with the assumption that $t \notin C$, ensures $a(t)$ is well-defined.  Similarly let $b(t)$ be the smallest element of $C \cap (t,\infty)$ if that intersection is nonempty, and $b(t)=+\infty$ otherwise.  Finally, if $t \in I \setminus C$, set $J(t):= \big( a(t), b(t) \big)$.

Let 
\[
\mathcal{J}:= \big\{ J(t) \ : \ t \in I \setminus C \big\},
\]
and notice $\mathcal{J}$ is a pairwise disjoint collection of open intervals; hence,
\begin{equation}\label{eq_J_ctble}
\mathcal{J} \text{ is countable.}
\end{equation}
Also, for each $J \in \mathcal{J}$, $\varphi|_{J}$ is an element of $\text{Homeo}^+(J)$, and moreover is fixed-point-free, since $J=(a(t),b(t))$ for some $t \in I \setminus C$, and $C \cap (a(t),b(t))$ is empty.  The following Claim is a basic application of the Intermediate Value Theorem:
\begin{claim}\label{clm_IVT}
Each non-identity member of $\langle \varphi \rangle$ has the same fixed points as $\varphi$.
\end{claim}
\begin{proof}
Suppose $\tau$ is a non-identity member of $\langle \varphi \rangle$; then $\tau = \varphi^n$ for some nonzero $n \in \mathbb{Z}$.  Clearly every fixed point of $\varphi$ is also a fixed point of $\tau$.  Now suppose $t \in I$ is \emph{not} a fixed point of $\varphi$. Consider the interval $J:=(a(t), b(t))$ and the restriction $\varphi|_J$.  Since $\varphi|_J$ has no fixed points in $J$, the Intermediate Value Theorem implies that either $\varphi(x) < x$ for all $x \in J$, or $\varphi(x) > x$ for all $x \in J$.  In the former case, since $\varphi$ is order preserving it follows by induction that $\varphi^n(x) < x$ and $\varphi^{-n}(x) > x$ for all $n \in \mathbb{N}$ and all $x \in J$ (and similarly for the latter case, with inequalities reversed).  In particular, $t$ is not a fixed point of $\tau$. 
\end{proof}

Suppose $J \in \mathcal{J}$.  Since $\varphi|_J$ is fixed-point-free, Claim \ref{clm_IVT} ensures that the cyclic subgroup $\langle \varphi|_J \rangle$ generated by $\varphi|_J$ in $\text{Homeo}^+(J)$ acts freely on $J$.  So by Theorem \ref{thm_MainPos}, there exists a good, $\langle \varphi|_J \rangle$-anonymous predictor $\mathcal{P}_J$ for functions in ${}^J S$.  Define
\[
\mathcal{Q}: {}^{\underset{\smile}{I}} S \to S
\]
as follows:  fix any $s_0 \in S$.  Given $f: (-\infty, t_f) \to S$, consider cases:
\begin{itemize}
 \item If $t_f \in C$, let $\mathcal{Q}(f):= s_0$.
 \item If $t_f \notin C$, define $\mathcal{Q}(f):= \mathcal{P}_{J(t_f)} \Big( f|_{J(t_f)}  \Big)$; note that the domain of $f|_{J(t_f)}$ is $J(t_f) \cap (-\infty, t_f) = \big( a(t_f), t_f \big)$.
\end{itemize}

To see that $\mathcal{Q}$ is $\langle \varphi \rangle$-anonymous, consider any $f:(-\infty,t_f) \to S$ and any non-identity $\tau \in \langle \varphi \rangle$.  We need to show that $\mathcal{Q}(f) = \mathcal{Q}\big( f \tau \big)$.  If $t_f \in C$, then by the claim, $t_f$ is a fixed point of both $\varphi$ and $\tau$, and hence $t_f = \tau^{-1}(t_f)$ is the right endpoint of the domain of $f \tau$.  So by definition of $\mathcal{Q}$, $\mathcal{Q}(f) = s_0 = \mathcal{Q}(f \tau)$.  If $t_f \notin C$, then since $a(t_f)$ and $b(t_f)$ are fixed points of $\tau$, $\tau^{-1}(t_f)$ is an element of $ J(t_f)$.  In particular,
\begin{equation}\label{eq_J_shift}
\tau^{-1}(t_f) \notin C, \ J\left(\tau^{-1}(t_f)\right) = J(t_f) =:J, \text{ and } (f \tau)|_{J} = f|_{J} \ \tau|_{J}.
\end{equation}
Then
\begin{align*}
\mathcal{Q} (f) & = \mathcal{P}_{J}\Big( f|_{J} \Big) & \text{(definition of $\mathcal{Q}$)} \\
& = \mathcal{P}_{J}\Big( f|_{J}  \  \tau|_{J} \Big) & \text{(by $\langle \varphi|_{J} \rangle$-anonymity of $\mathcal{P}_{J}$)} \\
& =  \mathcal{P}_J \Big( (f\tau)_J \Big) & \text{(by \eqref{eq_J_shift})} \\
& = \mathcal{Q}\big( f \tau \big). & \text{(by definition of $\mathcal{Q}$, since $\tau^{-1}(t_f) \notin C$)}
\end{align*}

To see that $\mathcal{Q}$ is good, fix any $H: I \to S$, and let
\[
B:= \Big\{ t \in I \ : \  \mathcal{Q} \left( H|_{(\infty,t)} \right) \ne H(t) \Big\}.
\]
Now $B = (B \cap C) \cup (B \setminus C)$, so it suffices to show that each set in the union is null.  Since $C$ is null by assumption, $B \cap C$ is null; so it remains to show that $B \setminus C$ is also null.  Now
\begin{equation}\label{eq_UnionMeasureZero}
B \setminus C \subseteq \bigcup_{J \in \mathcal{J}} (B \cap J). 
\end{equation}
Moreover, by the way we defined $\mathcal{Q}$, $B \cap J$ is exactly the set of $t \in J$ where $\mathcal{P}_J\big( H|_{J \cap (-\infty, t)} \big) \ne H(t)$, which has measure zero because $\mathcal{P}_J$ is good.  Together with \eqref{eq_J_ctble}, this implies that the right side of \eqref{eq_UnionMeasureZero} has measure zero.

\begin{remark}
Suppose $U \le \text{Homeo}^+(I)$, and there is a $C \subset I$ such that $C$ is closed (in $I$), $C$ has measure zero, and each member of $C$ is a fixed point of every member of $U$.  Closure of $C$ ensures we can define $\mathcal{J}=\left\{ \big(a(t),b(t)\big) \ : \ t \in I \setminus C \right\}$  as in the proof above.  Then for each $J \in \mathcal{J}$ and $\varphi \in U$, $\varphi|_J$ is a member of $\text{Homeo}^+(J)$, and $U_J:= \{ \varphi|_{J} \ : \ \varphi \in U \}$ is a subgroup of $\text{Homeo}^+(J)$.  \textbf{If} there exists a good, $U_J$-anonymous predictor for every $J \in \mathcal{J}$, then an argument similar to the one above allows us to amalgamate them to yield a good, $U$-anonymous predictor for functions on $I$. 
\end{remark}

%% file: NegativeResult.tex
\section{Proof of Theorem \ref{thm_MainNeg}}\label{sec_Neg}

In this section we prove Theorem \ref{thm_MainNeg}, which strengthens Bajpai-Velleman's Theorem \ref{thm_BV_neg}.  We remark that their proof actually showed something a bit stronger than is stated in Theorem \ref{thm_BV_neg} (or in their paper).  Namely, their set $S=\mathbb{R}/\sim$ had the property that there is no good $S$-predictor that is anonymous with respect to the class
\begin{equation}\label{eq_CinftyLip}
C^\infty_{\text{Lipschitz}} \cap \text{Homeo}^+(\mathbb{R})
\end{equation}
where $C^\infty_{\text{Lipschitz}}$ is the set of $f \in C^\infty$ such that for all $k \in \mathbb{N}$, $f^{(k)}$ is bounded.

We strengthen their result in two ways:
\begin{enumerate}
 \item We show that their result still holds when the class \eqref{eq_CinftyLip} is replaced by the smaller class $D^\infty_{\text{Lipschitz}}$ of infinitely Lipschitz \emph{diffeomorphisms}; i.e., the set of invertible $f: \mathbb{R} \to \mathbb{R}$ such that both $f$ and $f^{-1}$ are $C^\infty$ functions, and for every $k \in \mathbb{N}$, $f^{(k)}$ and $(f^{-1})^{(k)}$ are bounded.  
 \item We show that even a weaker kind of  prediction fails.  
\end{enumerate} 
 
We first describe the weaker kind of prediction.  Given a set $S$, let $\boldsymbol{{}^{\mathbb{R}^{\circ}} S}$ denote the set of $S$-valued functions $f$ such that $\text{dom}(f) = \mathbb{R} \setminus \{ h_f \}$ for some $h_f \in \mathbb{R}$ (so the domain of $f$ is the reals with a hole at $h_f$).  If $F$ is a total function on $\mathbb{R}$ and $x \in \mathbb{R}$, $F \setminus \{ x \}$ will denote the restriction of $F$ to the domain $\mathbb{R} \setminus \{x \}$.  A \textbf{weak $S$-predictor} will refer to any function
\[
\mathcal{P}: {}^{\mathbb{R}^{\circ}} S \to S,
\]
and it will be called a \textbf{good} weak $S$-predictor if for every total $F: \mathbb{R} \to S$, the set
\[
\Big\{ x \in \mathbb{R} \ : \  \mathcal{P}\left( F \setminus \{ x \} \right)  \ne F(x)  \Big\}
\]
has Lebesgue measure zero.  Roughly, $\mathcal{P}$ almost always ``predicts" $F(x)$ based on the past \emph{and} future behavior of $F$.

Every good $S$-predictor yields a good weak $S$-predictor, because if $\mathcal{P}$ is a good $S$-predictor, define $\mathcal{Q}$ on ${}^{\mathbb{R}^{\circ}} S$ as follows:  given any $f \in {}^{\mathbb{R}^{\circ}} S$ whose domain has a hole at $h_f$, define
\[
\mathcal{Q}(f):= \mathcal{P}\Big( f \restriction (-\infty, h_f)  \Big).
\]
Then for any total $F: \mathbb{R} \to S$ we have $\mathcal{Q}\Big( F \setminus \{ x \}  \Big) = \mathcal{P} \Big( F \restriction (-\infty,x)  \Big)$, which can only fail to equal $F(x)$ on a measure zero set. 

If $U \le \text{Homeo}^+(\mathbb{R})$, let us say that a weak $S$-predictor $\mathcal{P}$ is \textbf{$\boldsymbol{U}$-anonymous} if whenever $f,g \in {}^{\mathbb{R}^{\circ}} S$ (with holes $h_f$, $h_g$ in their respective domains), $\varphi \in U$, $\varphi(h_f) = h_g$, and $f = g \circ \varphi \restriction \left( \mathbb{R} \setminus \{ h_f \} \right)$, then $\mathcal{P}(f) = \mathcal{P}(g)$.  Similarly as above, any good $U$-anonymous $S$-predictor yields a good $U$-anonymous weak $S$-predictor.

Our goal is to prove:

\begin{theorem}\label{thm_MainNegativeWeak}
There is an equivalence relation $\sim$ on $\mathbb{R}$ such that, letting $S= \mathbb{R}/\sim$,  there is no good, $D^\infty_{\text{Lipschitz}}$-anonymous, weak $S$-predictor.
\end{theorem}

Our proof of Theorem \ref{thm_MainNegativeWeak} relies heavily on the proof of \cite[Theorem 8]{MR3552748}.  Much of the following lemma was implicit in their proof.

\begin{lemma}\label{lem_BlockingEquivRel}
Suppose $U \subseteq \text{Homeo}^+(\mathbb{R})$ and $\sim$ is an equivalence relation on $\mathbb{R}$ such that:
\begin{enumerate}
 \item No equivalence class has full measure;\footnote{A set $X$ has full measure if $\mathbb{R} \setminus X$ has Lebesgue measure zero.} and
 \item For every $P=(x,y) \in \mathbb{R}^2$ there exists a $\varphi \in U$ such that:
 \begin{enumerate}
  \item $\varphi(x)=y$; and
  \item if $z \ne x$ then $z \sim \varphi(z)$; i.e., each equivalence class is closed under $\varphi \restriction \left( \mathbb{R} \setminus \{x \} \right)$.
 \end{enumerate}
\end{enumerate}
Then there is no $U$-anonymous, good, weak $\mathbb{R}/\sim$-predictor.  In fact, the particular function $x \mapsto [x]_\sim$ witnesses the non-goodness of \textbf{every} $U$-anonymous, weak $\mathbb{R}/\sim$-predictor.
\end{lemma}
\begin{proof}
Let $S:=\mathbb{R}/\sim$, and suppose toward a contradiction that $\mathcal{P}:{}^{\mathbb{R}^{\circ}} S \to S$ is a $U$-anonymous, good, weak $S$-predictor.  Let $E: \mathbb{R} \to S$ be defined by $E(x) = [x]_\sim$.  We will get a contradiction by proving that $\mathcal{P}\Big( E \setminus \{ x \}  \Big) \ne E(x)$ holds for positively many $x$.  It suffices to prove that 
\begin{equation}\label{eq_EquivAnyPair}
\forall x,y \in \mathbb{R} \ \ \mathcal{P}\Big( E \setminus \{ x \}  \Big) = \mathcal{P} \Big( E \setminus \{ y \} \Big), 
\end{equation}
because if the equivalence class $C$ were their constant value, then the equation 
\begin{equation}\label{eq_PredictAtxE}
\mathcal{P}\Big( E \setminus \{ x \} \Big) = E(x) \ \Big( = [x]_\sim \Big)
\end{equation}
could hold only if $x \in C$; and since $C$ does not have full measure, this would yield positively many $x$ where the equation \eqref{eq_PredictAtxE} fails.

To prove \eqref{eq_EquivAnyPair} pick any $x,y \in \mathbb{R}$. By assumption, there is a $\varphi \in U$ such that $\varphi(x)=y$ and $z \sim \varphi(z)$ whenever $z \ne x$.  It follows that $E \setminus \{x \} = E \setminus \{ y \} \circ \varphi \restriction \Big(\mathbb{R} \setminus \{ x \} \Big)$.  Then by $U$-anonymity of $\mathcal{P}$, $\mathcal{P}\Big( E \setminus \{ x \}  \Big) = \mathcal{P} \Big( E \setminus \{ y \} \Big)$.
\end{proof}

\begin{corollary}\label{cor_BlockingFunctions}
Suppose $U \subseteq \text{Homeo}^+(\mathbb{R})$ and $\mathcal{F}$ is a family such that:
\begin{enumerate}
 \item $\mathcal{F}$ is a countable set of partial, injective functions from $\mathbb{R} \to \mathbb{R}$;
 \item\label{item_ThruEveryPoint} For every $(x,y) \in \mathbb{R}^2$ there is some $\varphi$ such that:
 \begin{enumerate}
  \item $\varphi \in U$;
  \item $\varphi(x)=y$; and
  \item If $z \ne x$, then there is some $f \in \mathcal{F}$ such that $\varphi(z)  = f(z)$.

\end{enumerate}  
\end{enumerate}

Let $\sim$ be the equivalence relation on $\mathbb{R}$ generated by the relation
\[
R:=\Big\{ (u,v) \ : \ \exists f \in \mathcal{F}  \ \ f(u)=v \Big\}.
\]
Then there is no good weak $U$-anonymous $\mathbb{R}/\sim$ predictor.
\end{corollary}
\begin{proof}
We verify that $\sim$ satisfies the assumptions of Lemma \ref{lem_BlockingEquivRel}.  Each $\sim$-equivalence class is countable, because $\mathcal{F}$ is countable and each member of $\mathcal{F}$ is injective (``countable-to-one" would suffice).  So each $\sim$-equivalence class not only fails to have full measure, but in fact has measure zero.  Consider any $(x,y) \in \mathbb{R}^2$. By assumption, there is a $\varphi \in U$ such that $\varphi(x)=y$ and whenever $z \ne x$, $\varphi(z) = f(z)$ for some $f \in \mathcal{F}$.  Then $(z,\varphi(z)) \in R$ and $R$ is contained in $\sim$, so $z \sim \varphi(z)$.
\end{proof}

We first sketch the outline of Bajpai-Velleman's proof of Theorem \ref{thm_BV_neg}.  They first fix a single, increasing $C^\infty$ bijection
\[
s: [0,1] \to [0,1]
\]
with vanishing derivatives at the endpoints (i.e., $s^{(k)}_-(0) = 0 = s^{(k)}_+(1)$ for all $k \in \mathbb{N}$, where the minus and plus denote left and right derivatives, respectively).  A \textbf{rational point} is a point in the plane such that both coordinates are rational.  By \textbf{rational line} we will mean a line with positive slope that passes through two distinct rational points.\footnote{Equivalently, a rational line is a line that has (positive) rational slope and passes through at least one rational point.  Note that there are only countably many rational lines, and that the intersection of two non-parallel rational lines is a rational point.}
If $\Psi:\mathbb{R} \to \mathbb{R}$ is a rational line (i.e., an affine function whose graph is a rational line), we write $s \circ \Psi$ to denote the obvious partial function with domain $\Psi^{-1} \Big( [0,1]  \Big)$.  They show that the countable collection
\begin{align}\label{eq_BV_F}
\begin{split}
\mathcal{F}:=   \Big\{ \Phi \circ s \circ \Psi \ :\  \Phi, \Psi \text{ are rational lines } \Big\}
\end{split}
\end{align}
witnesses the assumptions of Corollary \ref{cor_BlockingFunctions}.  This basically amounts to showing that, given any $(x,y) \in \mathbb{R}^2$, one can form an infinite concatenation of members of $\mathcal{F}$ to yield a strictly increasing $C^\infty$ function $h: (-\infty,x) \to \mathbb{R}$, whose approach (from the left) to the point $(x,y)$ is flat enough that $\hat{h}:=h ^\frown (x,y)$ will still be a $C^\infty$ function from $(-\infty,x] \to \mathbb{R}$, with $\hat{h}^{(k)}_-(x) = 0$ for all $k \in \mathbb{N}$.  Then extending $\hat{h}$ to all of $\mathbb{R}$ (in a $C^\infty$ manner) is easy.

We follow a similar strategy to prove Theorem \ref{thm_MainNegativeWeak}, though we must use a different collection in order to avoid vanishing first derivatives.  Fix any ``bump" function
\[
b: [0,1] \to \mathbb{R}
\]
with the following properties:
\begin{enumerate}[label=(\roman*)]
 \item $b(0) = 0 = b(1)$;
 \item $b(x) > 0$ for all $x \in (0,1)$;
 \item $b \in C^\infty$ (with one-sided derivatives taken at 0 and 1) and $b^{(k)}(0)=0=b^{(k)}(1)$ for all $k \in \mathbb{N}$; and
 \item $\int_{[0,1]} b = 1$.

\end{enumerate}

To get such a $b$, one could start, for example, with the common bump function
\[
\Psi(x)=
\begin{cases}
 \exp \left( \frac{-1}{1-x^2}\right) & \text{ if } x \in (-1,1) \\
 0 & \text{ otherwise,}
\end{cases}
\]
let $\Phi(x)=\Psi(2x-1)$, and then define $b:[0,1] \to \mathbb{R}$ by $b(x) = \frac{\Phi(x)}{\int_{[0,1]} \Phi}$. 

\begin{definition}\label{def_GlueFcn}
Given any point $A=(p,q)$, any $\Delta>0$ and any $\gamma > 0$, define
\[
F_{A,\Delta,\gamma}: [p,p + \Delta] \to [q, q + \Delta(1 + \gamma) ]
\]
by
\[
F_{A,\Delta,\gamma}(x)=q + \Delta \int_{\left[ 0, \frac{x-p}{\Delta} \right]} (1 + \gamma b).
 \]
\end{definition}

$F_{A,\Delta,\gamma}$ will serve as a smooth transition function from the point $A=(p,q)$ to the point $\big(p+\Delta, q + \Delta(1 + \gamma) \big)$, with derivative 1 (and vanishing higher derivatives) at those endpoints (see Figure \ref{fig_F_ADeltaGamma}).  The $\Delta$ and $\gamma$ parameters are used to control the norm of the higher derivatives in our subsequent construction.  

\begin{figure}[h]
\caption{The function $F_{A,\Delta,\gamma}$}
\label{fig_F_ADeltaGamma}
\centering
\includegraphics[scale=.4]{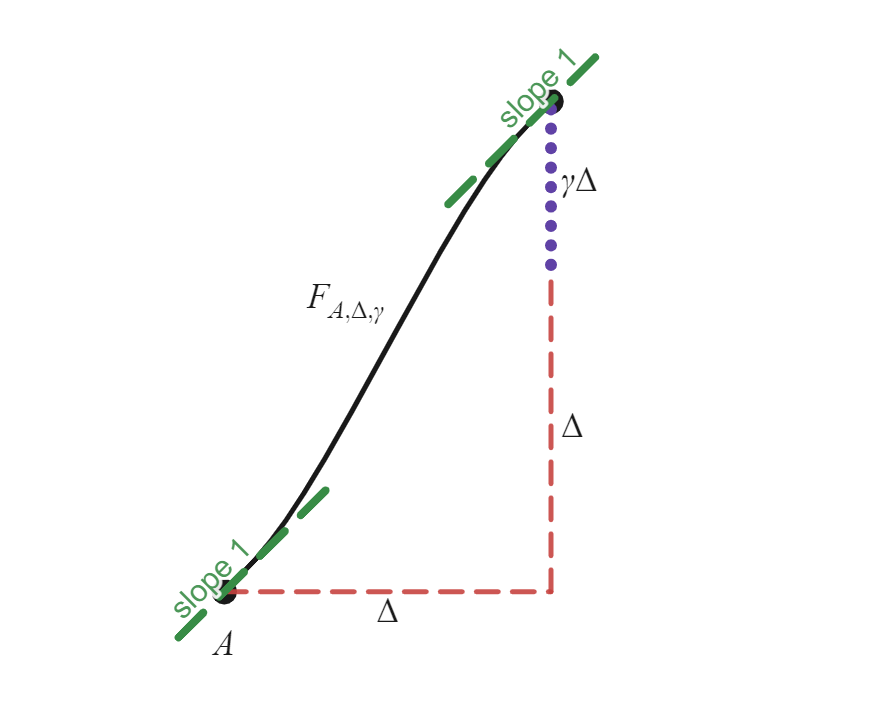}
\end{figure}

The key features are:
\begin{enumerate}[label=(\Roman*)]
 \item\label{item_Endpoints} $F_{A,\Delta,\gamma}(p) = q$ and $F_{A,\Delta,\gamma}(p+\Delta) = q + \Delta(1+\gamma)$;

 \item\label{item_Fprime} The Fundamental Theorem of Calculus and Chain Rule yield, for all $x \in [p,p + \Delta]$:
 \[
 F_{A,\Delta,\gamma}'(x) = \Delta \left( 1 + \gamma b\left( \frac{x-p}{\Delta} \right)   \right)\left( \frac{1}{\Delta}\right) =1 + \gamma b\left( \frac{x-p}{\Delta} \right).
 \]
In particular, since $\gamma > 0$ and $b > 0$ on $(0,1)$, $F_{A,\Delta,\gamma}$ is increasing.  Together with \ref{item_Endpoints}, it follows that $F_{A,\Delta,\gamma}$ is an increasing bijection from $[p,p+\Delta]$ to $[q,q+\Delta(1+\gamma)]$.  Also, notice that for all $x \in [p,p + \Delta]$: 
\[
 \lvert F'_{A,\Delta,\gamma}(x)-1 \rvert \le \gamma \lVert b \rVert,
\]
 where $\lVert \cdot \rVert$ denotes the sup-norm.

 \item\label{item_DerivAtEndpoints} Since $b(0)=0=b(1)$, $F_{A,\Delta,\gamma}'(p) = 1 = F_{A,\Delta,\gamma}'(p + \Delta)$; 
 
 \item\label{item_HigherDerivsBound} Following part \ref{item_Fprime}, we see that
 \[
\forall k \ge 2 \ \  F^{(k)}_{A,\Delta, \gamma}(x)= \frac{\gamma}{\Delta^{k-1}} b^{(k-1)}\left( \frac{x-p}{\Delta} \right)
 \] 
 so, in particular,
 \begin{equation*}
 \forall k \ge 2 \ \  \lVert F^{(k)}_{A,\Delta,\gamma} \rVert \le \frac{\gamma}{\Delta^{k-1}} \lVert b^{(k-1)} \rVert. 
 \end{equation*}

\end{enumerate}

Let $\mathcal{F}$ denote the collection of all rational lines of positive slope, together with all functions of the form $F_{A,\Delta,\gamma}$ where $A$ is a rational point and $\Delta$ and $\gamma$ are positive rational numbers.   We will show that $\mathcal{F}$ satisfies the assumptions of Corollary \ref{cor_BlockingFunctions}, with $U:= D^\infty_{\text{Lipschitz}}$.  Clearly $\mathcal{F}$ is countable, and consists of injective functions.  It remains to verify clause \ref{item_ThruEveryPoint} of Corollary \ref{cor_BlockingFunctions}. 

\begin{lemma}\label{lem_RationalPointsApproac}
Given any point $P=(w,z)$ in the plane, there exists sequences $A_n=(p_n,q_n)_n$ and $(\gamma_n)_n$ such that:
\begin{enumerate}
 \item $(p_n)_n$ and $(q_n)_n$ are increasing sequences of rational numbers converging to $w$ and $z$, respectively.
 \item $\gamma_n$ is a positive rational number for all $n$.
  \item\label{item_DeltanLess1} Let $\Delta_n:= p_{n+1}-p_n$.   Then $0 < \Delta_n < 1$ and $q_{n+1} = q_n + \Delta_n(1 + \gamma_n)$ for all $n$.
  \item\label{item_LimitQuotZero} $\lim_{n \to \infty} \frac{\gamma_n}{\left(\Delta_n\right)^{n-1}} = 0$.
\end{enumerate}
\end{lemma}
\begin{proof}
Fix an increasing sequence $(p_n)_n$ of rational numbers converging to $w$ such that $\Delta_n := p_{n+1}-p_n < 1$ for all $n$.  Let $L$ be the line of slope 1 that passes through $P$.\footnote{Notice that, unless $P$ happens to be a rational point, $L$ will not be a rational line.}  For any point $B$, let $\text{vDist}(L,B)$ denote the vertical distance from $L$ to $B$.  

Recursively define the $q_n$, $A_n$, and $\gamma_n$, together with the auxilliary variable
\[
v_n:= \text{vDist}(L,A_n),
\]
 as follows:  fix a rational $q_1$ such that $A_1:=(p_1,q_1)$ lies strictly below the line $L$, and
\begin{equation}\label{eq_BaseRecursInd}
v_1 < \Delta_1.
\end{equation}
Given that $q_n$ is defined (and hence so are $A_n=(p_n,q_n)$ and $v_n = \text{vDist}(L,A_n)$), let $\gamma_n$ be a positive rational number such that 
\begin{equation}\label{eq_IneqDeltanplus1}
v_n - \frac{\left( \Delta_{n+1}\right)^{n+1}}{n+1} < \gamma_n \Delta_n < v_n;
\end{equation}
this is possible because the $\Delta$'s are positive.  Define $q_{n+1}:=q_n + \Delta_n(1 + \gamma_n)$, $A_{n+1} := (p_{n+1},q_{n+1})$, and $v_{n+1}:= \text{vDist}(L,A_{n+1})$.  Note that $q_{n+1}$ is rational, provided that $q_n$ was rational.

We inductively verify that 
\begin{equation}
v_n < \frac{(\Delta_n)^n}{n} \tag{\text{IH}$_n$}
\end{equation}
holds for all $n$.  It holds for $n=1$ by \eqref{eq_BaseRecursInd}.  Now suppose $\text{IH}_n$ holds.  Since $L$ has slope 1, it follows from the definition of $A_{n+1}$ that $v_{n+1} = v_n - \gamma_n \Delta_n$, which by \eqref{eq_IneqDeltanplus1} is less than $\frac{(\Delta_{n+1})^{n+1}}{n+1}$; so (IH$_{n+1}$) holds.

Inequalities (IH$_n$) and \eqref{eq_IneqDeltanplus1} yield 
\[
\gamma_n \Delta_n < v_n < \frac{\left( \Delta_n \right)^n}{n}, 
\]
and hence
\[
\frac{\gamma_n}{\left( \Delta_n \right)^{n-1}} < \frac{1}{n}
\]
for all $n$.  This ensures part \eqref{item_LimitQuotZero} of the lemma.
\end{proof}

Fix any point $P=(w,z)$ in the plane, and let $A_n=(p_n,q_n)_n$, $\Delta_n=p_{n+1}-p_n$, and $\gamma_n$ be rational numbers as given by Lemma \ref{lem_RationalPointsApproac}.  For each $n$, let
\[
F_n:=F_{A_n, \Delta_n, \gamma_n}: [p_n, p_{n+1}] \to [q_n, \underbrace{q_n + \Delta_n(1+\gamma_n)}_{q_{n+1}}]
\]
using notation from Definition \ref{def_GlueFcn}.  See Figure \ref{fig_PieceTogether}.  

\begin{figure}[h]
\caption{The $F_n$ functions}
\label{fig_PieceTogether}
\centering
\includegraphics[scale=.4]{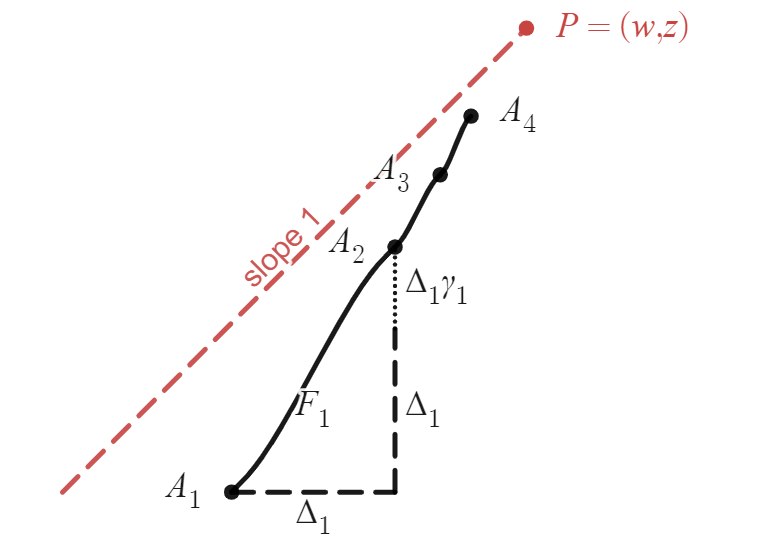}
\end{figure}

Let $F:= \bigcup_n F_n$, which maps from $[p_1,w) \to \mathbb{R}$.  By \ref{item_DerivAtEndpoints} on page \pageref{item_DerivAtEndpoints}, $F_n$ and $F_{n+1}$ meet at the point $A_{n+1}$, both with first derivative 1 and vanishing higher derivatives; in particular,
\begin{equation}\label{eq_C_infty_before_w}
F \text{ is a } C^\infty \text{ function on } [p_1,w) \text{ and } \lim_{x \nearrow w} F(x) = z.  
\end{equation}

By \ref{item_Fprime}, we have the following for all $n \in \mathbb{N}$:

\begin{equation}\label{eq_BoundFirstDeriv}
\forall x \in [p_n,\underbrace{p_n + \Delta_n}_{p_{n+1}} ] \  \lvert F'_{n}(x)-1 \rvert \le \gamma_n \lVert b \rVert.
\end{equation}
Since $\gamma_n \to 0$ as $n \to \infty$ (by requirements   \eqref{item_DeltanLess1} and \eqref{item_LimitQuotZero} of Lemma \ref{lem_RationalPointsApproac}), it follows that 
\begin{equation}\label{eq_LimitFirstDeriv}
\lim_{x \nearrow w} F'(x) = 1.
\end{equation}

By \ref{item_HigherDerivsBound}, the following also holds for all $n \in \mathbb{N}$, where $\alpha_{k,n}$ is defined as indicated:
\begin{equation}\label{eq_HigherDerivBoundDefAlpha}
\forall k \ge 2  \ \ \lVert F^{(k)}_{n} \rVert \le \underbrace{ \frac{\gamma_n}{(\Delta_n)^{k-1}} }_{\alpha_{k,n}} \lVert b^{(k-1)} \rVert. 
\end{equation}


Requirement \ref{item_DeltanLess1} of Lemma \ref{lem_RationalPointsApproac} (that $0 < \Delta_n < 1$) ensures that 
\[
n \ge k \ \implies \ 0 <  \underbrace{\frac{\gamma_n}{(\Delta_n)^{k-1}}}_{\alpha_{k,n}} \le  \frac{\gamma_n}{(\Delta_n)^{n-1}},
\]
the last term of which goes to 0 as $n \to \infty$ by requirement \ref{item_LimitQuotZero} of Lemma \ref{lem_RationalPointsApproac}.  So
\[
\lim_{n \to \infty} \alpha_{n,k} = 0,
\]
and together with \eqref{eq_HigherDerivBoundDefAlpha} this yields
\begin{equation}\label{eq_LimitHigherDerivZero}
\forall k \ge 2 \ \ \lim_{n \to \infty} \lVert F^{(k)}_n \rVert =0.
\end{equation}

This entire construction, starting from Lemma \ref{lem_RationalPointsApproac}, has a straightforward analogue ``from the right" of the point $P=(w,z)$, yielding a \emph{decreasing} sequence $(p^R_{n})_n$ of rationals converging to $w$, and functions
\[
F^R_n: [p^R_{n+1}, p^R_n] \to \mathbb{R}
\]
such that $F^R_n \in \mathcal{F}$, and $F^R:= \bigcup_n F^R_n$ is an increasing $C^\infty$ function from $(w,p_1^R] \to \mathbb{R}$ with nonvanishing first derivative that has the following analogues of the properties above:
\begin{equation*}
\lim_{x \searrow w} F^R(x) = z, \ \lim_{x \searrow w} (F^R)'(x) = 1, \text{ and } \forall k \ge 2 \ \ \lim_{n \to \infty} \lVert (F^R_n)^{(k)} \rVert =0.
\end{equation*}

Let $\Psi^L$ be the affine function with slope 1 passing through the rational point $(p_1,q_1)$, $\Psi^R$ be the affine function with slope 1 passing through the rational point $(p^R_1, q^R_1)$, and define $G:\mathbb{R} \to \mathbb{R}$ by
\[
G(x)=
\begin{cases}
\Psi^L(x) & \text{ if } x < p_1 \\
F(x) & \text{ if } x \in [p_1,w) \\
z & \text{ if } x=w \\
F^R(x) & \text{ if } x  \in (w,p_1^R] \\
\Psi^R(x) & \text{ if } x > p_1^R. 
\end{cases}
\]

Then $G$ is continuous, invertible, and has the property that for every $x \ne w$, there is some $f \in \mathcal{F}$ such that $G(x) = f(x)$.\footnote{This $f$ is either $\Psi^L$ (if $x < p_1$), or one of the $F_n$'s (if $x \in [p_1,w)$), or one of the the $F^R_n$'s (if $x \in (w,p_1^R]$), or $\Psi^R$ (if $x > p_1^R$).}

It remains to prove that $G \in D^\infty_{\text{Lipschitz}}$.  Now $G$ is affine on $(-\infty,p_1]$ and on $[p_1^R,\infty)$.  In particular, 
\begin{equation}\label{eq_TailsBoundedDer}
G |_{(-\infty,p_1] \cup [p_1^R,\infty)} \text{ has bounded derivatives of all orders.}
\end{equation}
So, to prove that $G \in D^\infty_{\text{Lipschitz}}$, it suffices to show that $G \in D^\infty$; i.e., that both $G$ and $G^{-1}$ are $C^\infty$ functions.\footnote{Once we know that both $G$ and $G^{-1}$ are $C^\infty$, then by the Heine-Cantor and Extreme Value Theorems, for each $k \in \mathbb{N}$ the functions $G^{(k)}$ and $(G^{-1})^{(k)}$ would be bounded on the closed intervals $[p_1, p^R_1]$ and $[q_1, q^R_1]$, respectively.  Together with \eqref{eq_TailsBoundedDer}, this implies $G \in D^\infty_{\text{Lipschitz}}$. }  Furthermore, by the Inverse Function Theorem, to prove that the invertible function $G$ is in $D^\infty$, it suffices to show that $G \in C^\infty$ and that $G'$ is never zero.  Since $F_n$ agrees with $F_{n+1}$ (and $F^R_n$ agrees with $F^R_{n+1}$) at all derivatives at their meeting points, and their first derivatives are always positive, the only nontrivial point to deal with is $w$; we will show that:
\begin{itemize}
 \item  $G'(w) = 1$; and
 \item $G^{(k)}(w) = 0$ for $k \ge 2$.  
\end{itemize} 
 We will prove these from the left of $w$; the proof from the right is similar.  The following lemma is a minor variant \cite[Lemma 7]{MR3552748}:
\begin{lemma}\label{lem_FillInRightDeriv}
Suppose $f$ is continuous on $(a,b]$, differentiable on $(a,b)$, and 
\[
\lim_{x \nearrow b} f'(x) = L.
\]
Then $f$ is left differentiable at $b$, and $f'_-(b) = L$.
\end{lemma} 
\begin{proof}
Suppose $(x_n)$ is an arbitrary sequence in $(a,b)$ that converges to $b$.  By the Mean Value Theorem, for each $n$ there is a $c_n \in (x_n,b)$ such that $f'(c_n) = \frac{f(b)-f(x_n)}{b-x_n}$.  By assumption, $\lim_{n \to \infty} f'(c_n) = L$, so $\lim_{n \to \infty} \frac{f(b)-f(x_n)}{b-x_n} = L$.  
\end{proof}

By \eqref{eq_C_infty_before_w}, $G$ is $C^\infty$ on $(-\infty,w)$, and we already noted that $G$ is continuous on $\mathbb{R}$.  By \eqref{eq_LimitFirstDeriv}, continuity of $G$, and Lemma \ref{lem_FillInRightDeriv}, $G'_-(w) = 1 = \lim_{x \nearrow w} G'(x)$ and hence $G'$ is continuous on $(-\infty,w]$.  Then, using \eqref{eq_LimitHigherDerivZero} (with $k=2$) and applying Lemma \ref{lem_FillInRightDeriv} to $G'$, we get that $G^{(2)}_-(w) = 0 = \lim_{x \nearrow w} G^{(2)}(x)$ and $G^{(2)}$ is continuous on $(-\infty,w]$.  Continuing by induction, we get that $G^{(k)}_-(w) = 0 = \lim_{x \nearrow w} G^{(k)}(x)$ for all $k \ge 2$.

%% file: Conclusion.tex
\section{Concluding Remarks}

Hardin-Taylor's Question \ref{q_BigQ} (page \pageref{q_BigQ}) is still very much open, as there is still a big gap between the ``lower frontier" of our Theorem \ref{thm_MainPos} and the ``upper frontier" of our Theorem \ref{thm_MainNeg}. 

Our Theorem \ref{thm_MainPos} isolated some (topological) group-theoretic properties that guarantee existence of anonymous predictors.  It is natural to pursue this further.  In particular, Theorem \ref{thm_FolkAppHolder} implies that any group $U \le \text{Homeo}^+(I)$ satisfying the assumptions of Theorem \ref{thm_MainPos} is a \textbf{metabelian} group, meaning that its commutator subgroup is abelian.  And Theorem \ref{thm_MainPos} encompasses all known examples of groups for which good, anonymous predictors exist (for every set $S$).  This suggests:
\begin{question}
Suppose $U \le \text{Homeo}^+(\mathbb{R})$, and for every set $S$ there exists a $U$-anonymous, good $S$-predictor.  Must $U$ be metabelian?  Must $U$ at least be solvable?
\end{question}

